\documentclass[a4paper]{article}
\usepackage{amsmath}
\usepackage{parskip}
\usepackage{authblk}
\usepackage[a4paper, total={4.9in, 8in}]{geometry}
\usepackage{amssymb}
\usepackage{amsfonts}
\usepackage{amsthm}
\usepackage{color}
\usepackage[hidelinks]{hyperref}
\usepackage{cite}

\title{On the Smallest Non-trivial Action of $\operatorname{SAut}(F_n)$ for Small $n$}
\author{Reemon Spector}
\affil{Hertford College, University of Oxford}
\date{Summer 2021}

\begin{document}

\theoremstyle{plain}
    \newtheorem{prop}{Proposition}[section]
    \newtheorem{lem}[prop]{Lemma}
    \newtheorem{cor}[prop]{Corollary}
    \newtheorem{thm}[prop]{Theorem}

\theoremstyle{definition}
    \newtheorem{defn}[prop]{Definition}
    \newtheorem{rem}[prop]{Remark}
    \newtheorem{ques}[prop]{Question}

\maketitle

\textbf{Abstract.} In this paper we investigate actions of $\operatorname{SAut}(F_n)$, the unique index $2$ subgroup of $\operatorname{Aut}(F_n)$, on small sets, improving upon results by Baumeister--Kielak--Pierro for several small values of $n$. Using a computational approach for $n \geqslant 5$, we show that every action of $\operatorname{SAut}(F_n)$ on a set containing fewer than $18$ elements is trivial.

\section{Introduction}
The actions of $\operatorname{Aut}(F_n)$, the automorphism group of the free group of rank $n$, on sets of small cardinality are not yet very well understood. In particular, the question of the cardinality of the smallest set, $\mathrm{X}$, that $\operatorname{SAut}(F_n)$ acts on nontrivially remains open.

In their paper on non-abelian quotients of $\operatorname{Aut}(F_n)$ \cite[Theorem 3.16]{BKP}, Bau\-meister, Kielak, and Pierro give a general lower bound for $|\mathrm{X}|$ when $n \geqslant 7$, as well as specific bounds for $n \in \{3,4,5,6\}$, with the bounds being sharp when $n \in \{3,4\}$. Further, using a result from Saunders' paper on permutation degrees for Coxeter groups \cite[Theorem 2.3]{Sau}, we can extract a bound of $2n$ for $n \geqslant 3$, which is a greater lower bound for the cases $n \in \{7,8\}$ than the one given in \cite[Theorem 3.16]{BKP}.

Using a computational proof by exhaustion, we prove \hyperref[n=5]{Theorem 3.10}, which states that any action of $\operatorname{SAut}(F_5)$ on a set containing fewer than $18$ elements is trivial. The best previous bound for $n = 5$ was $12$. \hyperref[n>5]{Corollary 3.12} then extends this result of $18$ to all $n \geqslant 5$, which improves on the previous bounds of $14$, $14$, $16$ for $n = 6, 7, 8$ respectively. Currently, the smallest known action coincides with the smallest action of $\operatorname{PSL_n}(\mathbb{Z}/2\mathbb{Z})$, which is on $2^n-1$ points for non-exceptional values of $n$ \cite{BKP}.

For a complete picture of the smallest actions of $\operatorname{SAut}(F_n)$, sharp results for the two cases $n \in \{1,2\}$ are proved in \hyperref[n=1]{Proposition 3.1} which states that $\operatorname{SAut}(F_1)$ is trivial, and therefore acts trivially on every set; and \hyperref[n=2]{Proposition 3.2} which states that $\operatorname{SAut}(F_2)$ has a nontrivial action on a set of two elements.

\textbf{Acknowledgements.} The author wishes to thank Dawid Kielak for his invaluable supervision, Hertford College and the Crankstart Bursary for their generous funding, the Oxford Mathematical Institute for access to powerful machinery, and Andres Klene-Sanchez for many a meaningful conversation.

\section{Background}
\subsection{\texorpdfstring{Subgroups of $\operatorname{SAut}(F_n)$}{Subgroups of SAut(Fn)}}
\label{subgroups}
We will start by constructing a few automorphisms of $F_n$, the free group of rank $n$, and the finite subgroups of $\operatorname{Aut}(F_n)$ generated by these automorphisms. First, pick an ordered basis $\mathrm{A} = \{a_1, a_2, \dots, a_n\}$ for $F_n$, and define $\mathrm{N} = \{1,2,\dots,n\}$.

For each pair of distinct $i,j \in \mathrm{N}$, define the \textit{transposition} $\sigma_{ij} \colon F_n \to F_n$ as
\[ \sigma_{ij}(a_k) = 
\begin{cases}
    a_j & \text{if }k = i,\\
    a_i & \text{if }k = j,\\
    a_k & \text{otherwise.}
\end{cases} \]
The transpositions $\sigma_{ij}$ generate a symmetric group $\mathrm{S}_n$ in $\operatorname{Aut}(F_n)$, which acts naturally on the labels of basis elements in $\mathrm{A}$.

In a similar fashion, for each $i \in \mathrm{N}$, define the \textit{involution} $\varepsilon_i \colon F_n \to F_n$ as
\[ \varepsilon_i(a_k) = 
\begin{cases}
    a_i^{-1} & \text{if }k = i,\\
    a_k & \text{otherwise.}
\end{cases} \]
For each pair of distinct $i,j \in \mathrm{N}$, the involutions $\varepsilon_{i}$ and $\varepsilon_{j}$ commute, and hence generate a copy of $(\mathbb{Z}/2\mathbb{Z})^n$ in $\operatorname{Aut}(F_n)$. Observe the conjugation action of $\mathrm{S}_n$ on $(\mathbb{Z}/2\mathbb{Z})^n$ naturally permutes the indices of the involutions, so preserves $(\mathbb{Z}/2\mathbb{Z})^n$. In fact, since $(\mathbb{Z}/2\mathbb{Z})^n$ intersects $\mathrm{S}_n$ trivially, this gives us a wreath product $(\mathbb{Z}/2\mathbb{Z})$ $\wr$ $\mathrm{S}_n$ in $\operatorname{Aut}(F_n)$. We will denote this group by $\mathrm{B}_n$ as it is isomorphic to the Coxeter group of type $B$ of rank $n$.

Next, we turn to the standard representation of $\mathrm{S}_{n+1}$ in $\operatorname{Aut}(F_n)$, which we can generate with the transpositions $\sigma_{ij}$ for $i,j \in \mathrm{N} \cup \{n+1\}$ where, for each $i \in \mathrm{N}$, the new transposition $\sigma_{i(n+1)} \colon F_n \to F_n$ is defined as
\[ \sigma_{i(n+1)}(a_k) = 
\begin{cases}
    a_i^{-1} & \text{if }k = i,\\
    a_ka_i^{-1} & \text{otherwise.}
\end{cases} \]

Observe that abelianizing $F_n$ induces a map $\phi \colon \operatorname{Aut}(F_n) \twoheadrightarrow \operatorname{GL}_n(\mathbb{Z})$, which we will refer to as the \textit{canonical map}. By associating the basis $\mathrm{A}$ with the standard basis for $\mathbb{Z}^n$, we see that
\begin{itemize}
    \item $\sigma_{ij}$ maps to the identity matrix with rows $i$ and $j$ swapped;
    \item $\varepsilon_i$ maps to the identity matrix with a $-1$ in position $(i,i)$; and
    \item $\sigma_{i(n+1)}$ maps to the identity matrix where row $i$ has all $-1$s.
\end{itemize}

With this perspective, it is natural to define $\operatorname{SAut}(F_n)$ as the kernel of the determinant map $\operatorname{det} \colon \operatorname{GL}_n(\mathbb{Z}) \xrightarrow[]{\operatorname{det}} \{\pm1\}$ composed with the canonical map $\phi$ from $\operatorname{Aut}(F_n)$.

We can now see $\mathrm{A}_n$, $\mathrm{A}_{n+1}$ and $(\mathbb{Z}/2\mathbb{Z})^{n-1}$ as subgroups of $\operatorname{SAut}(F_n)$, where $\mathrm{A}_{k} = \mathrm{S}_{k} \cap \operatorname{SAut}(F_n)$ and $(\mathbb{Z}/2\mathbb{Z})^{n-1} = (\mathbb{Z}/2\mathbb{Z})^n \cap \operatorname{SAut}(F_n)$. In addition, since $(\mathbb{Z}/2\mathbb{Z})^{n-1}$ is generated by pairs of involutions $\varepsilon_{i}\varepsilon_{j}$ and $\mathrm{A}_n$ acts naturally on the indices $i$ and $j$, we also see that $(\mathbb{Z}/2\mathbb{Z})^{n-1} \rtimes \mathrm{A}_n \leqslant \operatorname{SAut}(F_n)$. We will denote this subgroup by $\mathrm{D}_n'$, as it is isomorphic to the derived subgroup of the Coxeter group of type $D$ of rank $n$. Note that $\mathrm{D}_n'$ and $\mathrm{B}_n$ act naturally on a set of $2n$ points (or $n$ signed points).

\subsection{\texorpdfstring{Gersten's Presentation of $\operatorname{SAut}(F_n)$}{Gersten's Presentation of SAut(Fn)}}
\label{gersten}
A key tool used throughout the paper is the following concise presentation of $\operatorname{SAut}(F_n)$. Following the convention in \cite{BKP}, for $n \geqslant 2$ and distinct $i,j \in \mathrm{N}$, we define the \textit{right transvection} $\rho_{ij} \colon F_n \to F_n$, and its analogous \textit{left transvection} $\lambda_{ij} \colon F_n \to F_n$ as
\[\rho_{ij}(a_k) = 
\begin{cases}
    a_ia_j & \text{if }k = i,\\
    a_k & \text{otherwise.}
\end{cases} \]
\[\lambda_{ij}(a_k) = 
\begin{cases}
    a_ja_i & \text{if }k = i,\\
    a_k & \text{otherwise.}
\end{cases} \]

In \cite{Ger}, Gersten proves that for $n \geqslant 3$ these transvections generate $\operatorname{SAut}(F_n)$ subject to four families of relations. We formulate Gersten's relations as below noting that each $\pm1$ can be taken to be either $1$ or $-1$ independently of the choice of the other.
 
Let $i,j,k \in \mathrm{N}$ be distinct, let $l \in \mathrm{N}$ be distinct from $i,k$, and define the commutator of $a$ and $b$ as $[a,b] = aba^{-1}b^{-1}$; then
\begin{itemize}
    \item[$(\mathrm{r}_1)$] $[\rho_{ij}^{\pm 1},\rho_{kl}^{\pm 1}] = [\lambda_{ij}^{\pm 1},\lambda_{kl}^{\pm 1}] = [\rho_{ij}^{\pm 1},\lambda_{kl}^{\pm 1}] = 1$;
    \item[$(\mathrm{r}_2)$] $[\rho_{ij}^{-1},\rho_{jk}^{-1}] = \rho_{ik}^{-1}$ and $[\lambda_{ij}^{-1},\lambda_{jk}^{-1}] = \lambda_{ik}^{-1}$;
    \item[$(\mathrm{r}_3)$] $[\rho_{ij}^{\pm 1},\lambda_{il}^{\pm 1}] = 1$; and
    \item[$(\mathrm{r}_4)$] $(\lambda_{ij} \lambda_{ji}^{-1} \rho_{ij})^4 = (\rho_{ij} \rho_{ji}^{-1} \lambda_{ij})^4 = 1$.
\end{itemize} 
Note that $(\mathrm{r}_4)$ can be demystified by observing that $\lambda_{ij} \lambda_{ji}^{-1} \rho_{ij} = \varepsilon_i\sigma_{ij}$. We also note that, as a consequence of the commutator relations above, $\operatorname{SAut}(F_n)$ is perfect for all $n \geqslant 3$.

\section{Results}
We preface the main results of the paper with a well-known result about $\operatorname{SAut}(F_1)$.
\begin{prop} $(n=1).$
    \label{n=1}
    $\operatorname{SAut}(F_1)$ is trivial.
\end{prop}

\begin{proof}
    Pick a generating element $a_1$ for $F_1$. The homomorphism $\varphi \colon F_1 \to \mathbb{Z}$ given by $\varphi(a_1) = 1$ is bijective, showing that $F_1 \cong \mathbb{Z}$ and hence that $\operatorname{Aut}(\mathbb{Z})$ consists of
    \begin{itemize}
        \item the trivial automorphism ($1 \mapsto 1$); and
        \item the negation automorphism ($1 \mapsto -1$), which has determinant $-1$.
    \end{itemize}
    Hence $\operatorname{SAut}(\mathbb{Z}) \cong \operatorname{SAut}(F_1)$ is trivial.
\end{proof}

Proposition 3.1 has the immediate corollary that $\operatorname{SAut}(F_1)$ acts trivially on any set.

\begin{prop} $(n=2).$
    \label{n=2}
    The smallest nontrivial action of $\operatorname{SAut}(F_2)$ is on a set of two elements.
\end{prop}

\begin{proof}
    For $n = 2$, Nielsen \cite{Nie} showed the canonical map $\operatorname{Aut}(F_n) \twoheadrightarrow \operatorname{GL}_n(\mathbb{Z})$ is in fact an isomorphism and hence $\operatorname{SAut}(F_2)$ is isomorphic to $\operatorname{SL}_2(\mathbb{Z})$, a very well-understood group with many interesting properties.
    
    \newcommand{\abcd}{\begin{pmatrix} a & b\\ c & d \end{pmatrix}}
    
    We take an example from Conrad's investigation of $\operatorname{SL}_2(\mathbb{Z})$ \cite[Example 2.5]{Con}, which draws a connection to the theory of modular forms. Denoting the $n^{\mathrm{th}}$ roots of unity by $\mathbb{C}_{n}$, we construct an action of $\operatorname{SAut}(F_2)$ on a set of two elements in two steps. First, consider the epimorphism $\chi \colon \operatorname{SL}_2(\mathbb{Z}) \to \mathbb{C}_{12}$ given by
    \[ \chi \colon \abcd \mapsto \operatorname{exp}\left(\dfrac{2\pi i}{12} ((1-c^2)(bd+3(c-1)d+c+3)+c(a+d-3))\right).\]
    Second, compose $\chi$ with the map $\psi \colon \mathbb{C}_{12} \twoheadrightarrow \mathbb{C}_{2}$ given by
    \[ \psi \colon z \mapsto z^6 \]
    to get a nontrivial homomorphism $\operatorname{SAut}(F_2) \cong \operatorname{SL}_2(\mathbb{Z}) \twoheadrightarrow \mathbb{C}_{2} \cong \mathbb{Z}/2\mathbb{Z}$.
\end{proof}

\begin{prop}$(n \in \{3,4\})$.
    \label{n=3,4}
    The smallest nontrivial actions of $\operatorname{SAut}(F_3)$ and $\operatorname{SAut}(F_4)$ are on sets of $7$ and $8$ elements respectively.
\end{prop}

\begin{proof}
    This result is a direct consequence of \cite[Lemmas 3.1 and 3.3]{BKP}.
\end{proof}

We now turn to some results that form the foundations upon which the computational aspect of the paper relies.

\begin{lem}
    \label{homomorphisms}
    Given a homomorphism $\alpha \colon \mathrm{D}_n' \to \mathrm{S}_m$ and some $\tau \in \mathrm{S}_m$, there is at most one homomorphism $\psi \colon \operatorname{SAut}(F_n) \to \mathrm{S}_m$ satisfying $\alpha = \psi|_{\mathrm{D}_n'}$ and $\tau = \psi(\rho_{12})$.
\end{lem}

Equivalently, any homomorphism $\psi \colon \operatorname{SAut}(F_n) \to \mathrm{S}_m$ is completely determined by the restriction $\psi|_{\mathrm{D}_n'}$ and the image $\psi(\rho_{12})$.

\begin{proof}
    Crucial observations from \cite[Lemma 2.3]{BKP}, originally due to Bridson--Vogtmann \cite{BV}, tell us that all transvections are conjugate in $\operatorname{SAut}(F_n)$. In particular, conjugating on the right and given distinct $i,j \in \mathrm{N}$ we have
    \[ \rho_{ij}^{\varepsilon_i\varepsilon_j} = \lambda_{ij}. \]
    We know $\mathrm{A}_n$ acts transitively on the pairs $(i,j)$, hence given $\sigma \in$ $\mathrm{A}_n$ (and excusing a slight abuse of notation) we also have
    \[ \rho_{ij}^{\sigma} = \rho_{\sigma(i)\sigma(j)}. \]
    As $\psi$ is a homomorphism, these conjugacies also hold in $\operatorname{im}(\psi)$. Hence, since $\varepsilon_i\varepsilon_j, \sigma \in$ $\mathrm{D}_n'$, we can generate the images of all transvections given $\psi(\rho_{12})$ and the images of elements in $\mathrm{D}_n'$. This concludes the proof as the transvections generate $\operatorname{SAut}(F_n)$.
\end{proof}

Recall from Gersten's presentation that $\operatorname{SAut}(F_n)$ is perfect for $n \geqslant 3$. Hence the image of $\operatorname{SAut}(F_n)$ under $\psi$ must sit inside $\mathrm{A}_m \leqslant \mathrm{S}_m$.

\begin{lem}
    \label{centraliser}
    Let $n \geqslant 5$ and let $\mathrm{S}$ be the subgroup of $\mathrm{D}_n'$ consisting of all elements of $\mathrm{D}_n'$ that fix $a_1$ and $a_2$. Denoting the set of all elements in $G$ that commute with some $x$ by $C_G(x)$, we have $\mathrm{S} = C_{\mathrm{D}_n'}(\rho_{12})$.
\end{lem}

\begin{proof}
    Let $\sigma \in \mathrm{S}$ and let $w$ be a fixed word $w = l_1l_2 \dots l_k$ in $F_n$. Recall both $\sigma$ and $\rho_{12}$ are automorphisms of $F_n$ so we have
    \[ \sigma\circ\rho_{12}(w) = \prod_{i=1}^k (\sigma\circ\rho_{12}(l_i)). \]
    For each $1 \leqslant i \leqslant k$, if $l_i \in \{a_1^{\pm1},a_2^{\pm1}\}$ then  $\sigma\circ\rho_{12}(l_i) = \rho_{12}(l_i) = \rho_{12}\circ\sigma(l_i)$. Otherwise, if $l_i \in \{a_3^{\pm1},\dots,a_n^{\pm1}\}$ then $\sigma\circ\rho_{12}(l_i) = \sigma(l_i)$. Note that $\sigma \in \mathrm{D}_n'$ so there exist unique $\tau \in \mathrm{A}_n$, $\xi \in (\mathbb{Z}/2\mathbb{Z})^{n-1}$ such that $\sigma = \tau\circ\xi$.
    
    Further, $\xi \colon \{a_3^{\pm1},\dots,a_n^{\pm1}\} \to \{a_3^{\pm1},\dots,a_n^{\pm1}\}$, so $\sigma(l_i) \in \{a_3^{\pm1},\dots,a_n^{\pm1}\}$ if and only if $\tau(l_i)$ lies in the set $\{a_3^{\pm1},\dots,a_n^{\pm1}\}$. Suppose it does not, then $\sigma^{-1}$ does not fix one of $a_1^{\pm1},a_2^{\pm1}$, which is a contradiction as $\mathrm{S}$ is a subgroup. Hence, $\sigma\circ\rho_{12}(l_i) = \sigma(l_i) = \rho_{12}\circ\sigma(l_i)$ and therefore $\sigma\circ\rho_{12}(w) = \rho_{12}\circ\sigma(w)$. This shows $\mathrm{S} \subseteq C_{\mathrm{D}_n'}(\rho_{12})$.
    
    Now let $g \in C_{\mathrm{D}_n'}(\rho_{12})$ and suppose $g(a_1) \not= a_1$, so we have
    \[ g\circ\rho_{12}(a_1) = g(a_1a_2) = g(a_1)g(a_2) = \rho_{12}\circ g(a_1).\]
    If $g(a_1) = a_1^{-1}$ then $a_1^{-1}g(a_2) = a_2^{-1}a_1^{-1}$, which contradicts $F_n$ being free as $g$ is a permutation of $\{a_1^{\pm1},a_2^{\pm1},\dots,a_n^{\pm1}\}$. Hence $g(a_1) \not= a_1^{\pm1}$, so $\rho_{12}\circ g(a_1) = g(a_1)$, which forces $g(a_2)$ to be trivial, contradicting $g$ being injective.
    
    Alternatively, if $g(a_1) = a_1$ and $g(a_2) \not= a_2$ then $g\circ\rho_{12}(a_1) = g(a_1a_2) = g(a_1)g(a_2) = a_1g(a_2)$, while $\rho_{12}\circ g(a_1) = \rho_{12}(a_1) = a_1a_2$, forcing $g(a_2) = a_2$, which is a clear contradiction. This shows $C_{\mathrm{D}_n'}(\rho_{12}) \subseteq \mathrm{S}$, which concludes the proof.
\end{proof}

\begin{cor}
    \label{imrho12}
    If $\psi \colon \operatorname{SAut}(F_n) \to \mathrm{A}_m$ is a homomorphism, then $\psi(\rho_{12})$ must commute with every element in $\psi(\mathrm{S})$ in the image, that is $\psi(\rho_{12}) \in C_{\mathrm{A}_m}(\psi(\mathrm{S}))$.
\end{cor}

\begin{defn}
    \label{classes}
    We say that two homomorphisms from a group $G$ to a permutation group $H$ are \textit{related} if they are equal up to conjugacy of $H$. We hence define the \textit{homomorphism classes} between $G$ and $H$ as the equivalence classes under relation.
\end{defn}

\begin{rem}
    \label{agree}
    Let $\psi$ be a homomorphism with domain $\operatorname{SAut}(F_n)$. If $\alpha = \psi|_{\mathrm{D}_n'}$ and $\beta = \psi|_{A_{n+1}}$, then $\alpha$ and $\beta$ agree when restricted to $\mathrm{A}_n$, that is $\alpha|_{\mathrm{A}_n} = \beta|_{\mathrm{A}_n}$. 
\end{rem}

Consequently, given a homomorphism $\alpha \colon \mathrm{D}_n' \to \mathrm{A}_m$, if there is no homomorphism class representative $\beta \colon \mathrm{A}_{n+1} \to \mathrm{A}_m$ such that $\alpha(\mathrm{A}_n)$ and $\beta(\mathrm{A}_n)$ are conjugate in $\mathrm{A}_m$, then $\alpha$ cannot be extended to a homomorphism with domain $\operatorname{SAut}(F_n)$. Note that this is a weaker condition than the one in \hyperref[agree]{Remark 3.8}, but we use this in our proof of \hyperref[n=5]{Theorem 3.10} as it is significantly faster to compute.

\begin{lem}
    \label{monomorphisms}
    Let $\psi \colon \operatorname{SAut}(F_5) \to \mathrm{A}_m$, where $m < 31$ and define $\alpha = \psi|_{\mathrm{D}_5'}$ and $\beta = \psi|_{\mathrm{A}_6}$. If $\alpha$ or $\beta$ is not injective, then $\psi$ is trivial.
\end{lem}

\begin{proof}
    Suppose $\alpha$ is not injective, so that $\operatorname{ker}(\alpha)$ is nontrivial. There exists some nontrivial $\xi \in \operatorname{ker}(\alpha)$ such that $\psi(\xi)$ is central in $\operatorname{im}(\psi)$. By \cite[Lemma 2.5]{BKP}, we know that $\psi$ factors through $\operatorname{PSL_5}(\mathbb{Z}/2\mathbb{Z})$ or is trivial, in which case we are done.
    
    It is a fact that $\operatorname{PSL_5}(\mathbb{Z}/2\mathbb{Z})$ is simple and has order that is divisible by 31. Hence $m \geqslant 31$ whenever $\psi$ is nontrivial, so the claim follows when $\alpha$ is not injective.
    
    Suppose $\beta$ is not injective. Since $\mathrm{A}_6$ is simple, this implies that $\beta$ is trivial. However, this forces $\alpha$ to be trivial on $\mathrm{A}_5$ by \hyperref[agree]{Remark 3.8}, hence making $\alpha$ non-injective, which completes the proof.
\end{proof}

\begin{thm}
    \label{n=5}
    Any action of $\operatorname{SAut}(F_5)$ on a set containing fewer than $18$ elements is trivial.
\end{thm}

\begin{proof}
    The proof for this Theorem relies entirely on a computer search using \cite{GAP}, the code for which is available at \cite{Spe}. Our strategy mainly revolves around using \hyperref[homomorphisms]{Lemma 3.4}, \hyperref[imrho12]{Corollary 3.6}, \hyperref[agree]{Remark 3.8} and \hyperref[monomorphisms]{Lemma 3.9} to attempt to extend a homomorphism from $\mathrm{D}_5'$ to one of $\operatorname{SAut}(F_5)$. The computation is split into three main phases, which we will denote (\texttt{P$_1$}), (\texttt{P$_2$}) and (\texttt{P$_3$}).
    
    In (\texttt{P$_1$}), we construct the automorphisms defined in \hyperref[subgroups]{Section 2.1} and use them to generate the subgroups $\mathrm{A}_n$, $\mathrm{A}_{n+1}$, $\mathrm{D}_n'$, and $\mathrm{S}$ in $\operatorname{SAut}(F_n)$.
    
    In (\texttt{P$_2$}), we construct the homomorphism classes from $\mathrm{D}_5'$ to $\mathrm{A}_{17}$, and from $\mathrm{A}_6$ to $\mathrm{A}_{17}$. We then use \hyperref[agree]{Remark 3.8} and \hyperref[monomorphisms]{Lemma 3.9} to narrow down the list of potential homomorphisms $\psi$ which we test in (\texttt{P$_3$}).
    
    In (\texttt{P$_3$}), we rely on \hyperref[homomorphisms]{Lemma 3.4} and \hyperref[imrho12]{Corollary 3.6} to test whether $(\mathrm{r}_2)$ from Gersten's presentation holds in the image of the homomorphism $\psi$ being tested. In particular, we test whether the following relation holds
    \[ [\psi(\rho_{12})^{-1} , \psi(\rho_{23})^{-1}] = \psi(\rho_{13})^{-1}. \]
    The program searches through all potential images $\psi(\rho_{12}) \in C_{\mathrm{A}_m}(\psi(\mathrm{S}))$  and shows the only image that passes the test above is trivial, thus proving our claim.
\end{proof}

It is worth noting that when comparing (\texttt{P$_1$}) with (\texttt{P$_2$}) and (\texttt{P$_3$}), the first phase is by far the shortest, requiring very low levels of CPU and memory usage. However, (\texttt{P$_2$}) requires a vast amount of memory with varying levels of CPU usage, while (\texttt{P$_3$}) requires little memory, but is highly CPU-intensive.

The expected runtime of the program is in the ballpark of several hours and required over 64GB of memory to run through and terminate. Carrying out these computations required access to powerful computers made available by the Oxford Mathematical Institute.

\begin{prop}
    \label{increasing}
    Let $n \geqslant 2$ and let $x_n$ be the cardinality of the smallest set on which $\operatorname{SAut}(F_n)$ acts nontrivially. The sequence $(x_n)_{n \geqslant 2}$ is increasing.
\end{prop}

\begin{proof}
    We will show this by proving $x_n \leqslant x_{n+1}$ for all $n$. By the sharpness of the results for $n \in \{2,3\}$, we know this holds for the base case $n=2$.
    
    Suppose $\psi \colon \operatorname{SAut}(F_{k+1}) \to \operatorname{Sym}(\mathrm{X})$ is given where $|\mathrm{X}| = x_{k+1}$ and let $\varphi = \psi|_{\operatorname{SAut}(F_k)}$. If $\varphi$ is trivial then pick an element $\xi$ in $\mathrm{A}_n < \mathrm{D}_n'$ so that $\varphi(\xi)$ is central in $\operatorname{im}(\psi)$, hence $\psi$ is trivial by \cite[Lemma 2.5 (3)]{BKP}, which is a  contradiction.
    
    Therefore $\varphi \colon \operatorname{SAut}(F_{k}) \to \operatorname{Sym}(\mathrm{X})$ is nontrivial, so $x_k \leqslant |\mathrm{X}| = x_{k+1}$ and the result follows by induction on $n$.
\end{proof}

\begin{cor}
    \label{n>5}
    Any action of $\operatorname{SAut}(F_n)$ on a set containing fewer than $18$ elements is trivial.
\end{cor}

\begin{proof}
    This is immediate by combining the results from \hyperref[n=5]{Theorem 3.10} and \hyperref[increasing]{Propo-} \hyperref[increasing]{sition 3.11}.
\end{proof}

\section{Closing Remarks}
As pointed out by the author's supervisor Dawid Kielak, one might conjecture that for sufficiently large $n$, the smallest action of $\operatorname{SAut}(F_n)$ is on a set of $m$ elements, where $m$ is the smallest integer such that there exist monomorphisms $\alpha \colon \mathrm{D}_n' \to \mathrm{A}_m$ and $\beta \colon \mathrm{A}_{n+1} \to \mathrm{A}_m$ satisfying $\alpha|_{\mathrm{A}_n} = \beta|_{\mathrm{A}_n}$. The inspiration behind this was the role played by $\mathrm{B}_n$ and $\mathrm{S}_{n+1}$ in rigidity; see \cite{BV}.

We also note that the kernel of the determinant map from $\mathrm{B}_n$ is a slightly larger subgroup, which we will denote by $\mathrm{K} < \operatorname{SAut}(F_n)$. In fact, it is the largest finite subgroup by the Nielsen Realisation Theorem for $\operatorname{Aut}(F_n)$, summarised by \cite{Vog} and originally due to \cite{Cul}, \cite{Zim} and \cite{Khr} independently. We conclude the paper with the following question, inspired by Kielak.

\begin{ques}
    For general $n$, does there exist an integer $m$ and monomorphisms $\beta \colon \mathrm{A}_{n+1}  \hookrightarrow \mathrm{A}_m$ and $\gamma \colon \mathrm{K} \hookrightarrow \mathrm{A}_m$ such that $\beta|_{\mathrm{A}_n} = \gamma|_{\mathrm{A}_n}$, but every action of $\operatorname{SAut}(F_n)$ on a set of $m$ elements is trivial?
\end{ques}

\end{document}